\documentclass{amsart}

\usepackage[T1]{fontenc}
\usepackage[latin1]{inputenc}
\usepackage{amsmath}
\usepackage{amssymb}
\usepackage{amsthm}
\usepackage{dsfont}

\theoremstyle{plain}\newtheorem{theo}{Theorem}
\theoremstyle{plain}\newtheorem{cor}{Corollary}
\theoremstyle{definition}
\theoremstyle{plain}\newtheorem{defi}{Definition}[section]
\theoremstyle{plain}\newtheorem{lem}[defi]{Lemma}
\theoremstyle{plain}
\theoremstyle{definition}

\newcommand{\N}{{\mathds N}}
\newcommand{\R}{{\mathds R}}
\newcommand{\Z}{{\mathds Z}}
\DeclareMathOperator{\var}{Var}
\DeclareMathOperator{\cov}{Cov}
\DeclareMathOperator{\sgn}{sgn}

\begin{document}

\title[Nonnormal Limits and Bootstrap for Quantiles]{Normal Limits, Nonnormal Limits, and the Bootstrap for Quantiles of Dependent Data}

\author[O.Sh. Sharipov, M. Wendler]{Olimjon Sh. Sharipov, Martin Wendler}

\email{Martin.Wendler@rub.de}



\begin{abstract} We will show under very weak conditions on differentiability and dependence that the central limit theorem for quantiles holds and that the block bootstrap is weakly consistent. Under slightly stronger conditions, the bootstrap is strongly consistent. Without the differentiability condition, quantiles might have a non-normal asymptotic distribution and the bootstrap might fail.
\end{abstract}

\maketitle
\section{Limit Behaviour of Quantiles}

Let $\left(X_n\right)_{n\in\Z}$ be a stationary sequence of real-valued random variables with distribution function $F$ and $p\in\left(0,1\right)$. Then the $p$-quantile $t_p$ of $F$ is defined as
\begin{equation*}
t_p:=F^{-1}\left(p\right):=\inf\left\{t\in\R\big|F\left(t\right)\geq p\right\}
\end{equation*}
and can be estimated by the empirical $p$-quantile, i.e. the $\lceil \frac{n}{p} \rceil$-th order statistic of the sample $X_1\ldots,X_n$. This also can be expressed as the $p$-quantile $F^{-1}_n\left(p\right)$ of the empirical distribution function $F_n\left(t\right):=\frac{1}{n}\sum_{i=1}^{n}\mathds{1}_{X_i\leq t}$. It is clear that $F^{-1}_n\left(p\right)$ is greater than $t_p$ iff $F_n\left(t_p\right)$ is smaller than $p$. This relation between the empirical quantile and the empirical distribution function was made more precise by Bahadur. Under the condition that the random variables $X_1,\ldots,X_n$ are independent and that $F$ is differentiable twice in a neighborhood of $t_p$, he proved that
\begin{equation*}
F^{-1}_n\left(p\right)-t_p=\frac{p-F_n\left(t_p\right)}{f\left(t_p\right)}+R_n,
\end{equation*}
where $f=F'$ is the derivative of the distribution function and $R=O(n^{-\frac{3}{4}}\log n)$ almost surely. Ghosh (1971) established a weak form of the Bahadur representation, only assuming that $F$ is differentiable once in $t_p$. He showed that $R_n=o_P\left(n^{-\frac{1}{2}}\right)$ under this condition, meaning that $\sqrt{n}R_n$ converges to 0 in probability. As noted by Lahiri (1992), the condition that $F$ is differentiable is also necessary for the central limit theorem for $F_n^{-1}(p)$. Weiss (1971) derived the nonnormal limit distribution of quantiles if the densitiy has a jump in $t_p$. De Haan and Taconis-Haantjes (1979) and Ghosh and Sukhatme (1981) investigated the asymptotic distribution for $F_n^{-1}(p)$, if $F$ is not differentiable, but regular varying. We will extend their results to strongly mixing random variables. 

There is a broad literature on the Bahadur representation for strongly mixing data. Babu and Singh (1978) proved such a representation under an exponentially fast decay of the strong mixing coefficients, this was weakened by Yoshihara (1995), Sun (2006) and Wendler (2011) to a polynomial decay of the strong mixing coefficients. All these articles deal with the case that $F$ is differentiable.

\begin{defi} Let $\left(X_n\right)_{n\in\Z}$ be a stationary process. Then the strong mixing coefficients are defined as
\begin{equation}
 \alpha (k) := \sup \left\{\left| P[AB] - P[A]P[B] \right| : A \in \mathcal{F}^n_1, B \in \mathcal{F}^\infty_{n+k}, n \in \Z \right\}
\end{equation}
where $\mathcal{F}^l_a$ is the $\sigma$-field generated by random variables $X_a, \ldots, X_l$. We say that $\left(X_n\right)_{n\in\mathds{Z}}$ is strongly mixing if $\lim_{k \rightarrow \infty} \alpha (k) = 0$.
\end{defi}
For further information on strong mixing and a detailed description of the other mixing assumptions, see Bradley (2007).

The paper is organized at follows: We will start with our asymptotic results for sample quantiles. In the next section, we will introduce the bootstrap procedure and give theorems about the consistency for the bootstrap. In the third section, we will provide the proofs of our results. Our first theorem is a nonlinear version of the weak Bahadur representation under mild mixing assumptions:

\begin{theo}\label{theo1} Let $\left(X_n\right)_{n\in\Z}$ be a stationary, strongly mixing sequence of random variables with distribution function $F$, such that for a $\rho>0$
\begin{equation*}
F(t_p+h)-F(t_p)=M|h|^\rho\sgn(h)+o(|h|^\rho)
\end{equation*}
as $h\rightarrow0$ and $\sum_{n=1}^\infty\alpha(n)<\infty$. Then
\begin{equation*}
F^{-1}_n\left(p\right)-t_p=\left(\frac{|p-F_n\left(t_p\right)|}{M}\right)^{\frac{1}{\rho}}\sgn\left(p-F_n(t_p)\right)+R_n,
\end{equation*}
where $R_n=o_P(n^{-\frac{1}{2\rho}})$.
\end{theo}
The mixing assumption in this Theorem is milder than in the article of Sun (2006), who assumed that $\alpha(n)=O(n^{-\alpha})$ for an $\alpha>10$, and in the article of Wendler (2011), who assumed that $\alpha>3$, but in these theorems, the convergence is almost surely. An easy example of distribution function of the above form is $F(t)=\frac{1}{2}|t|^\rho\sgn(t)+\frac{1}{2}$ for $t\in[-1,1]$ and $p=\frac{1}{2}$.

By Theorem 1.6 of Ibragimov (1962), the summability of the mixing coefficients is a sufficient condition for the central limit theorem of partial sums of bounded random variables, so we obtain with the help of the above representation:

\begin{theo}\label{theo1b} If the assumptions of Theorem \ref{theo1} hold, then $n^{\frac{1}{2\rho}}(F^{-1}_n\left(p\right)-t_p)$ converges in distribution to $C|W|^{\frac{1}{\rho}}\sgn(W)$, where $C$ is a constant and $W$ is a normal distributed random variable.
\end{theo}

For the special case $\rho=1$ (differentiability), we get Lemma 5.1 of Sun and Lahiri (2006) (central limit theorem for $F^{-1}_n\left(p\right)$), as $|W|\sgn(W)=W$ is normally distributed. Note that $\sum_{n=1}^\infty\alpha(n)<\infty$ is the best known condition for the central limit theorem for bounded random variables and that the central limit theorem might fail if one just assumes that $\alpha(n)=O(\frac{1}{n})$, see remark 10.11 in the book of Bradley (2007). So it is not possible to establish asymptotic normality of the sample quantile under this condition.

In the other case $\rho\neq1$, the limit $C|W|^{\frac{1}{\rho}}\sgn(W)$ has a nonnormal distribution.

\section{Block Bootstrap for Quantiles} The statistical inference for quantiles is a dificult task, many methods rely on estimates of the unknown density. An alternative method is the Bootstrap. Bickel and Freedman (1981) established the consistency of the Bootstrap for quantiles for independent data, more work on this topic was done by Ghosh et. al (1984) and Babu (1986).

For dependent data, normal approximation becomes even more difficult, but there is up to our knowledge only one article about the bootstrap for quantiles under dependence: Sun and Lahiri (2006) have shown the strong consistency (almost sure convergence to the right limit) of the bootstrap under strong mixing. We will establish a weak Bahadur representation for the bootstrap version of the quantile and will conclude that the bootstrap is weakly consistent for $\rho=1$, that means the bootstrap distribution function converges in probabilty to the same limit as the true distribution function. The bootstrap is not consistent for $\rho\neq1$.

There are different ways to resample blocks, for example the circular block bootstrap or the moving block bootstrap (for a detailed description of the different bootstrapping methods see Lahiri (2003)). We consider the circular block bootstrap introduced by Politis and Romano (1992). Instead of the original sample of n observations with an unknown distribution, construct new samples $X^\star_{1},\ldots,X^\star_{bl}$ as follows: Extend the sample $X_{1},\ldots,X_{n}$ periodically by $X_{i+n}=X_{i}$, choose blocks of $l=l_n$ consecutive observations of the sample randomly and repeat that $b=\lfloor\frac{n}{l}\rfloor$ times independently: For $j=1,\ldots,n$, $k=0,\ldots,b-1$
\begin{equation*}
P^\star\left(X^\star_{kl+1}=X_{j},\ldots,X^\star_{(k+1)l}=X_{j+l-1}\right)=\frac{1}{n},
\end{equation*}
where $P^{\star}$ is the bootstrap distribution conditionally on $\left(X_n\right)_{n\in\mathbb{N}}$, $E^{\star}$ and $\operatorname{Var}^{\star}$ are the conditional expectation and variance. For the circular block bootstrap version of the sample mean, Radulovi\'{c} (1996) has established weak consistency under very weak conditions.

$F_n^\star(t)=\frac{1}{bl}\sum_{i=1}^{bl}\mathds{1}_{\{X_i^\star\leq t\}}$ denotes the Bootstrap version of the empirical distribution function and $F_n^{\star-1}(p)$ the $p$-quantile of the bootstrap sample.

\begin{theo}\label{theo2} Let $\left(X_n\right)_{n\in\Z}$ be a stationary, strongly mixing sequence of random variables with distribution function $F$, such that for a $\rho>0$ and $M\neq0$
\begin{equation*}
F(t_p+h)-F(t_p)=M|h|^\rho\sgn(h)+o(h)
\end{equation*}
as $h\rightarrow0$ and $\sum_{n=1}^\infty\alpha(n)<\infty$. Furthermore, choose the block length in such a way that $\frac{1}{l}+\frac{l}{n}\rightarrow0$. Then
\begin{equation*}
F^{\star-1}_n\left(p\right)-t_p=\left(\frac{|p-F_n^\star\left(t_p\right)|}{M}\right)^{\frac{1}{\rho}}\sgn(p-F_n^\star\left(t_p\right))+R_n^\star,
\end{equation*}
where $R_n^\star=o_P(n^{-\frac{1}{2\rho}})$.
\end{theo}

Note that we do not center $F^{\star-1}_n\left(p\right)$ with respect to the bootstrapped expectation, but with respect to the true quantile $t_p$. With the help of this theorem, we get weak consistency (convergence in probability) in the case $\rho=1$. For $\rho\neq1$ we will get inconsistency.

\begin{cor}\label{cor3} If the assumptions of Theorem \ref{theo2} hold with $\rho=1$ and additionally $\lim_{n\rightarrow\infty}\var[\sqrt{n}F_n(t_p)]>0$, then
\begin{equation*}
\sup_{t\in\R}\left|P^\star\left(F^{\star-1}_n\left(p\right)-F^{-1}_n\left(p\right)\leq t\right)-P\left(F^{-1}_n\left(p\right)-t_p\leq t\right)\right|\xrightarrow{n\rightarrow\infty}0
\end{equation*}
in probability.
\end{cor}

The assumptions of this corollary cannot be weakend: It is well known that $l\rightarrow\infty$ and $n/l\rightarrow\infty$ are necessary for bootstrap consistency of the bootstrap. The mixing rate is the best known rate for the central limit theorem. As asymptotic normality might still hold for the bootstrap under slower mixing rates (see Peligrad (1998)), the bootstrap might be inconsistent. Finally, if $\rho\neq1$, the bootstrap also fails:

\begin{cor}\label{cor4} If the assumptions of Theorem \ref{theo2} hold with $\rho\neq 1$ and additionally $\lim_{n\rightarrow\infty}\var[\sqrt{n}F_n(t_p)]>0$, then
\begin{equation*}
\sup_{t\in\R}\left|P^\star\left(F^{\star-1}_n\left(p\right)-F^{-1}_n\left(p\right)\leq t\right)-P\left(F^{-1}_n\left(p\right)-t_p\leq t\right)\right|\xrightarrow{n\rightarrow\infty}Z_\rho
\end{equation*}
in distribution, where $Z_\rho$ is a non-degenerate (non-constant) random variable.
\end{cor}
The reason for the inconsistency in the case $\rho\neq 1$ is that the shape of the distribution function $F$ at $F^{-1}_n\left(p\right)$ (the centering for the bootstrap quantile) is different from the shape of $F$ at $t_p$ (the centering for the empirical quantile). If $F$ is differentiable ($\rho=1$), the shape at the two points is asymptotically the same, as $F^{-1}_n\left(p\right)\rightarrow t_p$.

We also want to establish the strong consistency and we need slightly stronger conditions on the mixing coefficients and the block length:
\begin{theo}\label{theo3} Let $\left(X_n\right)_{n\in\Z}$ be a stationary, strongly mixing sequence of random variables with distribution function $F$ which is differentiable in $t_p$ with positive derivative $f$. We assume that the mixing coefficients satisfy $\alpha(n)=O(n^{-1-\epsilon})$  for an $\epsilon>0$. Furthermore, choose the block length in such a way that for some constants $C_1,C_2,\epsilon_1>0$
\begin{equation*}
C_1n^{\epsilon_1}\leq l_n \leq C_2n^{1-\epsilon_1}
\end{equation*}
and for all $k\in\N$
\begin{equation*}
l_{2^k}=l_{2^k+1}=\ldots=l_{2^{k+1}-1}.
\end{equation*}
Then
\begin{equation*}
F^{\star-1}_n\left(p\right)-F_n^{-1}(p)=\frac{F_n(t_p)-F_n^\star\left(t_p\right)}{f(t_p)}+R_n^\star,
\end{equation*}
where $R_n^\star=o_P^\star(n^{-\frac{1}{2}})$ almost surely.
\end{theo}

With this Bahadur-Ghosh representation and Theorem 2.4 of Shao and Yu (1993), the strong consistency of the bootstrap follows easily:

\begin{cor}\label{cor5} If the assumptions of Theorem \ref{theo3} hold and $\lim_{n\rightarrow\infty}\var[\sqrt{n}F_n(t_p)]>0$, then
\begin{equation*}
\sup_{t\in\R}\left|P^\star\left(F^{\star-1}_n\left(p\right)-F^{-1}_n\left(p\right)\leq t\right)-P\left(F^{-1}_n\left(p\right)-t_p\leq t\right)\right|\xrightarrow{n\rightarrow\infty}0
\end{equation*}
almost surely.
\end{cor}
Compared to Theorem 3.1 of Sun and Lahiri (2006), our assumptions on the mixing coefficient $\alpha(n)$, on the distribution function $F$ and on the block length $l$ are weaker. They assume that $\alpha(n)=O(n^{-\alpha})$ for an $\alpha>9,5$, that $l=o(n^{\frac{1}{2}})$ and that $F$ is continuously differentiable in a neighborhood of $t_p$.

\section{Proofs}

In the proofs, $C$ denotes an arbitrary constant, which may have different values from line to line and may depend on several other values, but not on $n\in\N$. We use the following lemma proved by Ghosh (1971):

\begin{lem}\label{lem1} Let $(V_n)_{n\in\N}$ and  $(W_n)_{n\in\N}$ be two sequences of random variables, such that
\begin{enumerate}
\item the sequence $(W_n)_{n\in\N}$ is tight,
\item For all $k\in\R$ and $\epsilon>0$
\begin{align*}
\lim_{n\rightarrow\infty}P\left(V_n\leq k, W_n\geq k+\epsilon\right)&=0\\
\lim_{n\rightarrow\infty}P\left(V_n \geq k+\epsilon, W_n\leq k\right)&=0.
\end{align*}
\end{enumerate}
Then $V_n-W_n\rightarrow0$ in probabality as $n\rightarrow\infty$.
\end{lem}

\begin{proof}[Proof of Theorem \ref{theo1}] The proof follows the ideas of Ghosh (1971). We set $g(x)=M|x|^\rho\sgn(x)$ and $W_n=g^{-1}(\sqrt{n}(p-F_n(t_p)))$. By Theorem 1.6 of Ibragimov (1962), $\sqrt{n}(p-F_n(t_p))$ converges to a normal limit and thus $(W_n)_{n\in\N}$ is tight. We define $V_n=n^{\frac{1}{2\rho}}(F^{-1}_n\left(p\right)-t_p)$ and $Z_{t,n}=g^{-1}(\sqrt{n}(F(t_p+\frac{t}{n^{1/2\rho}})-F_n(t_p+\frac{t}{n^{1/2\rho}})))$. We have that by the definition of the generalized inverse
\begin{equation*}
\left\{V_n\leq t\right\}=\left\{p\leq F_n(t_p+\frac{t}{n^{1/2\rho}})\right\}=\left\{Z_{t,n}\leq t_n\right\}
\end{equation*}
where $t_n:=g^{-1}(\sqrt{n}(F(t_p+\frac{t}{n^{1/2\rho}})-p))$. By our assumptions on $F$, for all $t\in\R$ we have that $t_n\rightarrow t$ as $n\rightarrow\infty$. We assumed that $\sum_{k=1}^\infty \alpha(k)<\infty$. By a well-known covariance inequality $\cov\left(\mathds{1}_{\{t_p<X_1\leq t_p+\frac{t}{n^{1/2\rho}}\}},\mathds{1}_{\{t_p<X_k\leq t_p+\frac{t}{n^{1/2\rho}}\}}\right)\leq 4\alpha(k-1)$, so we have that
\begin{multline*}
E\left(\sqrt{n}\left(p-F_n(t_p)-F(t_p+\frac{t}{n^{1/2\rho}})+F_n(t_p+\frac{t}{n^{1/2\rho}})\right)\right)^2\\
\leq 2\sum_{k=1}^\infty\left|\cov\left(\mathds{1}_{\{t_p<X_1\leq t_p+\frac{t}{n^{1/2\rho}}\}},\mathds{1}_{\{t_p<X_k\leq t_p+\frac{t}{n^{1/2\rho}}\}}\right)\right|\\
\shoveleft\leq 2\sum_{k=1}^{\lfloor n^{\frac{1}{4}}\rfloor}\var\left(\mathds{1}_{\{t_p<X_1\leq t_p+\frac{t}{n^{1/2\rho}}\}}\right)+8\sum_{k=\lfloor n^{\frac{1}{4}}\rfloor}^\infty\alpha(k)\\
\leq 2n^{\frac{1}{4}}\left|F(t_p+\frac{t}{n^{1/2\rho}})-p\right|+8\sum_{k=\lfloor n^{\frac{1}{4}}\rfloor}^\infty\alpha(k)\xrightarrow{n\rightarrow\infty}0,
\end{multline*}
so $\sqrt{n}(p-F_n(t_p)-F(t_p+\frac{t}{n^{1/2\rho}})+F_n(t_p+\frac{t}{n^{1/2\rho}}))\rightarrow0$ in probability and consequently $Z_{t,n}-W_n\rightarrow0$ in probability as $n\rightarrow\infty$. So
\begin{multline*}
\lim_{n\rightarrow\infty}P\left(V_n\leq t, W_n\geq t+\epsilon\right)=\lim_{n\rightarrow\infty}P\left(Z_{t,n}\leq t_n, W_n\geq t+\epsilon\right)\\
\leq \lim_{n\rightarrow\infty}P\left(Z_{t,n}\leq t+\frac{\epsilon}{2}, W_n\geq t+\epsilon\right)=0.
\end{multline*}
Lemma \ref{lem1} completes the proof.
\end{proof}

We omit the proof of Theorem \ref{theo1b}, as we think it is an obvious consequence of Theorem \ref{theo1} and the continuous mapping theorem.

\begin{proof}[Proof of Theorem \ref{theo2}] We will use a similar method as in the proof of Theorem \ref{theo1}. We define
\begin{align*}
W_n^\star&:=g^{-1}\left(\sqrt{bl}(p-F_n^\star(t_p))\right),\\
V_n^\star&:=(bl)^{\frac{1}{2\rho}}(F^{-1}_n\left(p\right)-t_p),\\
Z_{t,n}^\star&:=g^{-1}\left(\sqrt{bl}(F(t_p+\frac{t}{n^{1/2\rho}})-F_n^\star(t_p+\frac{t}{n^{1/2\rho}}))\right)
\end{align*}
Note that $\sqrt{bl}(F_n(t_p)-F_n^\star(t_p))$ converges to a normal limit by Theorem 2 of Radulovi\'{c} (1996), so the sequence $(\sqrt{bl}((p-F_n(t_p))+(F_n(t_p)-F_n^\star(t_p)))_{n\in\N}$ is tight and consequently the sequence $(W_n^\star)_{n\in\N}$ is also tight. It remains to show for any $t\in\R$ that $Z_{t,n}^\star-W_n^\star\rightarrow0$ in probability as $n\rightarrow\infty$. By the construction of the circular block bootstrap $E^\star F_n^\star(t)=F_n(t)$, so
\begin{align*}
&EE^\star\left(\sqrt{bl}\left(F(t_p+\frac{t}{n^{1/2\rho}})-F_n^\star(t_p+\frac{t}{n^{1/2\rho}})-p+F_n^\star(t_p)\right)\right)^2\\
=&EE^\star\left(\sqrt{bl}\left(F_n(t_p+\frac{t}{n^{1/2\rho}})-F_n^\star(t_p+\frac{t}{n^{1/2\rho}})-F_n(t_p)+F_n^\star(t_p)\right)\right)^2\\
&\ \ \ +E\left(\sqrt{bl}\left(F(t_p+\frac{t}{n^{1/2\rho}})-F_n(t_p+\frac{t}{n^{1/2\rho}})-p+F_n(t_p)\right)\right)^2.
\end{align*}
In the proof of Theorem \ref{theo1}, we have already shown that the second summand converges to zero. For the first summand, we conclude from the conditional independence of the resampled blocks and the definition of empirical distribution function
\begin{multline*}
EE^\star\left(\sqrt{bl}\left(F_n(t_p+\frac{t}{n^{1/2\rho}})-F_n^\star(t_p+\frac{t}{n^{1/2\rho}})-F_n(t_p)+F_n^\star(t_p)\right)\right)^2\\
=lEE^\star\left(F_n(t_p+\frac{t}{n^{1/2\rho}})-\frac{1}{l}\sum_{i=1}^l\mathds{1}_{\{X_i^\star\leq t_p+\frac{t}{n^{1/2\rho}}\}}-F_n(t_p)+\frac{1}{l}\sum_{i=1}^l\mathds{1}_{\{X_i^\star\leq t_p\}}\right)^2.
\end{multline*}
With probability $\frac{1}{n}$, we have $(X_1^\star,\ldots,X_l^\star)=(X_{j+1},\ldots,X_{j+l})$ (with $X_{i}=X_{i-n}$ for $i>n$), so
\begin{multline*}
lEE^\star\left(F_n(t_p+\frac{t}{n^{1/2\rho}})-\frac{1}{l}\sum_{i=1}^l\mathds{1}_{\{X_i^\star\leq t_p+\frac{t}{n^{1/2\rho}}\}}-F_n(t_p)+\frac{1}{l}\sum_{i=1}^l\mathds{1}_{\{X_i^\star\leq t_p\}}\right)^2\\
=\frac{l}{n}\sum_{j=1}^nE\left(F_n(t_p+\frac{t}{n^{1/2\rho}})-\frac{1}{l}\sum_{i=1}^l\mathds{1}_{\{X_{j+1}\leq t_p+\frac{t}{n^{1/2\rho}}\}}-F_n(t_p)+\frac{1}{l}\sum_{i=1}^l\mathds{1}_{\{X_{j+i}\leq t_p\}}\right)^2\\
\shoveleft\leq 2lE\left(F_n(t_p+\frac{t}{n^{1/2\rho}})-F(t_p+\frac{t}{n^{1/2\rho}})-F_n(t_p)+F(t_p)\right)^2\\
+2lE\left(F_l(t_p+\frac{t}{n^{1/2\rho}})-F(t_p+\frac{t}{n^{1/2\rho}})-F_l(t_p)+F(t_p)\right)^2.
\end{multline*}
These two summands converge to 0 as in the proof of Theorem \ref{theo1}, which completes the proof.
\end{proof}

\begin{proof}[Proof of Corollary \ref{cor3}] By Theorem 2 of Radulovi\'{c} (1996)
\begin{equation*}
\sup_{t\in\R}\left|P^\star\left(\sqrt{n}(F_n(t_p)-p)\leq t\right)-P\left(Y\leq t\right)\right|\xrightarrow{n\rightarrow\infty}0
\end{equation*}
and
\begin{equation*}
\sup_{t\in\R}\left|P^\star\left(\sqrt{bl}(F_n^\star(t_p)-F_n(t_p))\leq t\right)-P\left(Y\leq t\right)\right|\xrightarrow{n\rightarrow\infty}0
\end{equation*}
in probability for some normal random variable $Y$. Furthermore by Theorems \ref{theo1} and \ref{theo2}
\begin{equation*}
F^{-1}_n\left(p\right)-t_p=\frac{p-F_n\left(t_p\right)}{M}+R_n
\end{equation*}
and
\begin{multline*}
F^{\star-1}_n\left(p\right)-F^{-1}_n\left(p\right)=\left(F^{\star-1}_n\left(p\right)-t_p\right)-\left(F^{-1}_n\left(p\right)-t_p\right)\\
=\frac{p-F_n^\star\left(t_p\right)}{M}-\frac{p-F_n\left(t_p\right)}{M}+R_n^\star-R_n=\frac{F_n\left(t_p\right)-F_n^\star\left(t_p\right)}{M}+R_n^\star-R_n,
\end{multline*}
where $R_n=o_P(n^{-\frac{1}{2}})$ and $R_n^\star=o_P(n^{-\frac{1}{2}})$. So we can conclude that
\begin{multline*}
\sup_{t\in\R}\left|P^\star\left(F^{\star-1}_n\left(p\right)-F^{-1}_n\left(p\right)\leq t\right)-P\left(F^{-1}_n\left(p\right)-t_p\leq t\right)\right|\\
\shoveleft\leq \sup_{t\in\R}\left|P^\star\left(\sqrt{bl}(F^{\star-1}_n\left(p\right)-F^{-1}_n\left(p\right))\leq t\right)-P\left(\frac{-Y}{M}\leq t\right)\right|\\
+\sup_{t\in\R}\left|P^\star\left(\sqrt{n}(F_n\left(p\right)-t_p)\leq t\right)-P\left(-\frac{Y}{M}\leq t\right)\right|\xrightarrow{n\rightarrow\infty}0
\end{multline*}
in probability.
\end{proof}

\begin{proof}[Proof of Corollary \ref{cor4}] By Theorems \ref{theo1} and \ref{theo2}, we have that
\begin{equation*}
F_n^{\star-1}(p)-F_n^{-1}(p)=g^{-1}(p-F_n^\star\left(t_p\right))-g^{-1}(p-F_n\left(t_p\right))+R_n+R_n^\star
\end{equation*}
with $R_n+R_n^\star=o_P(n^{-\frac{1}{2\rho}})$, so
\begin{multline*}
\sup_{t\in\R}\left|P^\star\left(n^{\frac{1}{2\rho}}(F_n^{\star-1}(p)-F_n^{-1}(p))\leq t\right)\right.\\
\left.-P^\star\left(g^{-1}(\sqrt{n}(p-F_n^\star\left(t_p\right)))-g^{-1}(\sqrt{n}(p-F_n\left(t_p\right)))\leq t\right)\right|\xrightarrow{n\rightarrow\infty}0
\end{multline*}
in probability. Furthermore
\begin{equation*}
\sup_{t\in\R}\left|P\left(n^{\frac{1}{2\rho}}(F_n^{-1}(p)-t_p)\leq t\right)-P\left(g^{-1}(\sqrt{n}(p-F_n\left(t_p\right)))\leq t\right)\right|\xrightarrow{n\rightarrow\infty}0.
\end{equation*}
So it follows that
\begin{multline*}
\lim_{n\rightarrow\infty}\sup_{t\in\R}\left|P^\star\left(n^{\frac{1}{2\rho}}(F_n^{\star-1}(p)-F_n^{-1}(p))\leq t\right)-P\left(n^{\frac{1}{2\rho}}(F_n^{-1}(p)-t_p)\leq t\right)\right|\\
=\lim_{n\rightarrow\infty}\sup_{t\in\R}\left| P^\star\left(g^{-1}(\sqrt{n}(p-F_n^\star\left(t_p\right)))-g^{-1}(\sqrt{n}(p-F_n\left(t_p\right)))\leq t\right)\right.\\
\left.-P\left(g^{-1}(\sqrt{n}(p-F_n\left(t_p\right)))\leq t\right) \right|
\end{multline*}
To investigate the latter, note that by Theorem 2 of Radulovic (1996), $\sqrt{n}(p-F_n\left(t_p\right))$ and $\sqrt{n}(F_n\left(t_p\right)-F_n^\star\left(t_p\right))$ converge in distribution to two independent normal random variables $W_1$ and $W_2$ with the same variance. So
\begin{multline*}
\sup_{t\in\R}\left| P^\star\left(g^{-1}(\sqrt{n}(p-F_n^\star\left(t_p\right)))-g^{-1}(\sqrt{n}(p-F_n\left(t_p\right)))\leq t\right)\right.\\
\left.-P\left(g^{-1}(\sqrt{n}(p-F_n\left(t_p\right)))\leq t\right) \right|\\
\xrightarrow{n\rightarrow\infty}\sup_{t\in\R}\left| P\left(g^{-1}(W_1+W_2)-g^{-1}(W_2)\leq t|W_2\right)-P\left(g^{-1}(W_1)\leq t\right) \right|=:Z_{\rho}
\end{multline*}
in distribution, where $W_1$ and $W_2$ are two independent normal random variables. As the functions $x\rightarrow g^{-1}(x+y)-g^{-1}(y)$ and $x\rightarrow g^{-1}(x)$ are not identical for $y\neq0$, the random variable $Z_\rho$ is not $0$.
\end{proof}

\begin{proof}[Proof of Theorem \ref{theo3}] We define $a_{n}=2^k$ for the $k\in\N$ such that $2^k\leq n < 2^{k+1}$.
\begin{align*}
\tilde{W}_n:=&\frac{1}{\sqrt{a_n}}\sum_{i=1}^n\left(F_n(t_p)-\mathds{1}_{\{X_i^\star\leq t_p\}}\right),\\
\tilde{Z}_{t,n}:=&\frac{1}{\sqrt{a_n}}\sum_{i=1}^n\left(F_n(t_p+\frac{t}{\sqrt{a_n}})-\mathds{1}_{\{X_i^\star\leq t_p+\frac{t}{\sqrt{a_n}}\}}-F_n(t_p)+\mathds{1}_{\{X_i^\star\leq t_p\}}\right).
\end{align*}
Following the arguments of the proof of Theorem \ref{theo1}, we only have to show that the sequence $(\tilde{W}_n)_{n\in\N}$ is tight and that $\tilde{Z}_{t,n}\rightarrow 0$ in bootstrap probability for all $t\in\R$ almost surely. By Theorem 2.4 of Shao and Yu (1993), $(\sqrt{n}(F_n(t_p)-F_n^\star(t_p)))_{n\in\N}$ is almost surely asymptotically normal and thus $(\tilde{W}_n)_{n\in\N}$ is tight.

First note that by the construction of the bootstrap random variables, the summands of $\tilde{Z}_{t,n}$ are independent conditional on $X_1,\ldots,X_n$ when the indices $i$ lie in different blocks. Additionally, the random variables are centered in bootstrap probability and the sequence of blocks are stationary for fixed $n$. So
\begin{multline*}
E^\star\left(\tilde{Z}_{t,n}\right)^2\\
=\lfloor \frac{n}{l_n}\rfloor\frac{1}{a_n} \var^\star \left[\sum_{i=1}^{l_n}\left(F_n(t_p+\frac{t}{\sqrt{a_n}})-\mathds{1}_{\{X_i^\star\leq t_p+\frac{t}{\sqrt{a_n}}\}}-F_n(t_p)+\mathds{1}_{\{X_i^\star\leq t_p\}}\right)\right].
\end{multline*}
Recall that the bootstrap random variables $X_1^\star,\ldots,X_l^\star$ take the values $X_{j},\ldots,X_{j+l-1}$ for $j=1,\ldots,n$ with probability $\frac{1}{n}$ and that we have to set $X_j=X_{j-n}$ for $j>n$.  So we have the following upper bound for the bootstrap variance:
\begin{multline*}
\var^\star \left[\sum_{i=1}^{l_n}\left(F_n(t_p+\frac{t}{\sqrt{a_n}})-\mathds{1}_{\{X_i^\star\leq t_p+\frac{t}{\sqrt{a_n}}\}}-F_n(t_p)+\mathds{1}_{\{X_i^\star\leq t_p\}}\right)\right]\\
\leq \frac{1}{n}\sum_{j=1}^n\max_{m=1,\ldots,l}\left(\sum_{i=j}^{j+m-1}\left(F_n(t_p+\frac{t}{\sqrt{a_n}})-\mathds{1}_{\{X_i\leq t_p+\frac{t}{\sqrt{a_n}}\}}-F_n(t_p)+\mathds{1}_{\{X_i\leq t_p\}}\right)\right)^2\\
\leq 2\frac{1}{n}\sum_{j=1}^n\max_{m=1,\ldots,l}\left(\sum_{i=j}^{j+m-1}\left(F(t_p+\frac{t}{\sqrt{a_n}})-\mathds{1}_{\{X_i\leq t_p+\frac{t}{\sqrt{a_n}}\}}-F(t_p)+\mathds{1}_{\{X_i\leq t_p\}}\right)\right)^2\\
+2l_n^2\left(F_n(t_p+\frac{t}{\sqrt{a_n}})-F(t_p+\frac{t}{\sqrt{a_n}})-F_n(t_p)+F(t_p)\right)^2.
\end{multline*}
To show the convergence of the bootstrap variance, we now need moment bounds for the maximum of the partial sums. By the inequality of Davydov (1970), we have that
\begin{multline*}
\cov\left(\mathds{1}_{\{X_1\leq t_p+\frac{t}{\sqrt{a_n}}\}}-\mathds{1}_{\{X_1\leq t_p\}},\mathds{1}_{\{X_{1+k}\leq t_p+\frac{t}{\sqrt{a_n}}\}}-\mathds{1}_{\{X_{1+k}\leq t_p\}}\right)\\
\leq C\alpha^\frac{2}{2+\epsilon}(k)(\frac{t}{\sqrt{a_n}})^{\frac{\epsilon}{2+\epsilon}},
\end{multline*}
as $|F(t_p)-F(t_p+h)|\leq C|h|$. By standard calculations
\begin{multline*}
E\left(\sum_{i=1}^{m}\left(F(t_p+\frac{t}{\sqrt{a_n}})-\mathds{1}_{\{X_i\leq t_p+\frac{t}{\sqrt{a_n}}\}}-F(t_p)+\mathds{1}_{\{X_i\leq t_p\}}\right)\right)^2\\
\leq 2m\sum_{k=1}^\infty C\alpha^\frac{2}{2+\epsilon}(k)(\frac{t}{\sqrt{a_n}})^{\frac{\epsilon}{2+\epsilon}}\leq Cma_n^{-\frac{\epsilon}{4+\epsilon}}.
\end{multline*}
We obtain the following maximal inequality by Theorem 3 of M\'oricz (1976)
\begin{multline*}
E\left(\max_{m=1,\ldots,l}\left|\sum_{i=1}^{m}\left(F(t_p+\frac{t}{\sqrt{a_n}})-\mathds{1}_{\{X_i\leq t_p+\frac{t}{\sqrt{a_n}}\}}-F(t_p)+\mathds{1}_{\{X_i\leq t_p\}}\right)\right|\right)^2\\
\leq C l\log^2 la_n^{-\frac{\epsilon}{4+\epsilon}}.
\end{multline*}
To simplify the notation, we set
\begin{equation*}
Y_n(i):=F(t_p+\frac{t}{\sqrt{a_n}})-\mathds{1}_{\{X_i\leq t_p+\frac{t}{\sqrt{a_n}}\}}-F(t_p)+\mathds{1}_{\{X_i\leq t_p\}}
\end{equation*}
By the Chebyshev inequality
\begin{multline*}
\sum_{k=1}^\infty P(\max_{n=2^k,\ldots,2^{k+1}-1}\left|\tilde{Z}_{t,n}\right|\geq \delta)\\
\shoveleft\leq \frac{1}{\delta^2}\sum_{k=1}^\infty E\left[\max_{n=2^k,\ldots,2^{k+1}-1}\lfloor \frac{n}{l_n}\rfloor\frac{1}{a_n} \frac{2}{n}\sum_{j=1}^n\left(\max_{m=1,\ldots,l}\sum_{i=j}^{j+m-1}Y_n(i)\right)^2\right]\\
\shoveright{+\frac{1}{\delta^2}\sum_{k=1}^\infty E\left[\max_{n=2^k,\ldots,2^{k+1}-1}\lfloor \frac{n}{l_n}\rfloor\frac{2}{a_n}l_n^2\left(\frac{1}{n}\sum_{i=1}^n Y_n(i)\right)^2\right]}\\
\shoveleft \leq \frac{1}{\delta^2}\sum_{k=1}^\infty 8\frac{1}{l_{2^k}}E\left(\max_{m=1,\ldots,l}\sum_{i=1}^{m}Y_n(i)\right)^2\\
\shoveright{+\frac{1}{\delta^2}\sum_{k=1}^\infty4\frac{l_{2^k}}{a_{2^k}^2}E\left(\max_{m=1,\ldots,2^{k+1}-1}\sum_{i=1}^{m}Y_n(i)\right)^2}\\
\leq C\sum_{k=1}^\infty \log^2 (l_{2^k}) a_{2^k}^{-\frac{\epsilon}{4+\epsilon}}+C\sum_{k=1}^\infty \log^2 (a_{2^k}) a_{2^k}^{-\frac{\epsilon}{4+\epsilon}}<\infty.
\end{multline*}
With the Borel-Cantelli-lemma, we have that $\tilde{Z}_{t,n}$ converges to 0 almost surely for all $t\in\R$ and the proof is complete.
\end{proof}

\section*{Acknowledgement}
The research was supported by the DFG Sonderforschungsbereich 823 (Collaborative Research Center) {\em Statistik nichtlinearer dynamischer Prozesse}. We thank the anonymous referees for their useful comments.

\end{document}